\documentclass[10pt]{amsart}
\usepackage{graphicx}
\usepackage{amssymb}

\usepackage{amsmath}

\usepackage{tikz}

\usepackage{url}
\usepackage{mathtools}

\usepackage{apptools}
 
\usepackage{ulem}
\usepackage{url}
\usepackage{tikz-cd}
\usepackage{apptools}

\usepackage{graphicx} 
\usepackage[margin=2cm]{geometry}
\usepackage[initials]{amsrefs}
\usepackage{amsmath}
\usepackage{amsfonts}
\usepackage{amsthm}
\usepackage{amssymb}
\usepackage{enumitem}
\usepackage{xcolor}
\usepackage{todonotes}
\usepackage[colorlinks,linkcolor=blue]{hyperref}
\usepackage{comment}

\usepackage{tikz}
\usetikzlibrary{calc}
\usepackage{todonotes}

\usepackage{comment} 
\newtheorem{theorem}{Theorem}[section]
\newtheorem{lemma}[theorem]{Lemma}

\theoremstyle{definition}
\newtheorem{definition}[theorem]{Definition}

\newtheorem{question}[theorem]{Question}
\theoremstyle{remark}

\numberwithin{equation}{section}

\theoremstyle{definition}

\theoremstyle{remark}

\newcommand{\T}{\mathbb{T}}
\newcommand{\D}{{\rm Diam}}
\newcommand{\SU}{{\rm Sup}}
\numberwithin{equation}{section}

\begin{document}

\title{An anti-classification theorem for minimal homeomorphisms on the torus}


\author{Bo Peng}
\address{}
\curraddr{}
\email{}
\thanks{}


\date{}

\dedicatory{}

\begin{abstract}
We show that it is impossible to classify topological conjugacy relation of minimal homeomorphisms on the torus by countable structures. 
\end{abstract}

\maketitle


\bibliographystyle{amsplain}
\section{Introduction}

A \textbf{topological dynamical system} is a pair $(X,f)$ where $X$ is a compact metric space and $f$ is a homeomorphism. Let $(X,f)$ and  $(Y,g)$ be two topological dynamical systems, an \textbf{isomorphism} from $(X,f)$ to $(Y,g)$ is a homeomorphism $h: X\rightarrow Y$ such that $h\circ f=g\circ h$. Topological dynamics study properties which are preserved under isomorphism. When the space $X$ is fixed, the isomorphism relation for topological dynamical systems on $X$ is equivalent to the conjugacy relation on the group ${\rm Homeo}(X)$. We also use the terminology \textbf{topologically conjugacy} to replace isomorphism when studying homeomorphisms on a fixed compact metric space.

In 1961, Smale (\cite{ICM}, \cite{Smale}) suggested classifying smooth and topological dynamical systems on manifolds by topological conjugacy. That is, given a manifold $M$, can we determine whether two $C^r$-diffeomorphisms $(0\leq r\leq \infty$) are topologically conjugate. When $r=0$, the problem is the same as classifying homeomorphisms on $M$. Smale suggested classifying this equivalence relation on a given manifold by "numerical and algebraic invariants" (for more details regarding the background, see \cite{Foremangorodiff}).

Earlier results of classification problems regarding topological dynamical systems especially minimal topological dynamical systems can be traced back to the late 19th century. Let $f$ be a homeomorphism of the circle and $F$ be its lift to $\mathbb{R}$. Let $x\in \mathbb{R}$, then define
$$
\rho(F)={\rm lim}_{n} \frac{F^n(x)-x}{n}.
$$

Poincar\'{e} observed that this limit always exist and is independent of the choice of $x$. Also, for two different lifts of $f$, the difference of the corresponding limits is an integer. Thus, we can write $\rho(f)$ for the positive fractional part of $\rho(F)$ for any lift $F$ of $f$. This is a well-defined number in $[0,1)$, namely the \textbf{rotation number}. Poincar\'{e} got the following celebrating classification result:

\begin{theorem}
    {\rm(Poincar\'{e})} Suppose $f$ is a homeomorphism of the circle with a dense orbit, then $f$ is topologically conjugate with the rotation $R_{\rho(f)}$.
\end{theorem}

A topological dynamical system is \textbf{minimal} if every orbit is dense. Thus, Poincar\'{e}'s rotation number completely classified minimal homeomorphisms on the circle. This is a numerical invariant which satisfies Smale's suggestion.

A natural question is: Can we generalize Poincar\'{e}'s result to other manifolds, in particular, the torus.   In \cite{Jager}, J\"ager left the following comment:  

\textit{It seems natural attempt to generalize Poincar\'{e}'s result to higher dimensions, however, so far no result in this direction exist.}

Since Poincar\'{e}'s result was proved more than 100 years ago, it is natural to predict the impossibility of this generalization. In other words, it might be impossible to classify minimal homeomorphisms on the torus by any numerical or even algebraic invariants suggested by Smale and proved by Poncar\'{e} on the circle. To prove this impossibility, we need to describe this problem in a rigorous mathematical language.

In the late 20th century, a complexity theory was developed by descriptive set theorists (see \cite{Gaobook}, \cite{Kecbook}) and by applying this theory, we can obtain anti-classification results for equivalence relations.

Let $E$ and $F$ be two equivalence relations on Polish spaces $X$ and $Y$, respectively. A \textbf{Borel reduction} from $E$ to $F$ is a Borel function $f: X\rightarrow Y$ such that for any $x_1,x_2\in X$,
$$
x_1Ex_2\,\, \mbox{if and only if}\,\,f(x_1)Ff(x_2).
$$
We write $E\leq_B F$ if there is a Borel reduction from $E$ to $F$. A Borel reduction can be regarded as a "push-forward" of classification from $E$ to $F$. If $E\leq_B F$, classifying $F$ would be considered harder as classifying $E$. For more details regarding the motivation of this theory, like why we need to require Borelness, see \cite{Foremansurvey}.

A very important class of equivalence relations are called \textbf{orbit equivalence relations}. Let $G$ be a Polish group acting continuously on a Polish space $X$, let $E^X_G$ be an equivalence relation defined on $X$, two points $x,y\in X$ are equivalent if there exists $ g\in G$ such that $gx=y$. An equivalence relation on a Polish space $X$ is called an $\textbf{orbit equivalence relation}$ if it is defined as $E^X_G$ for some Polish group $G$.  A \textbf{complete orbit equivalence relation}  is an orbit equivalence relation such that all Polish group actions are Borel reducible to it. Affine homeomorphism relation of Choquet simplices and homeomorphism relation of compact metric spaces are complete orbit equivalence relations (see \cite{Sabok}, \cite{Zielinski}). 


Now we talk about three common benchmarks measuring the complexity of an equivalence relation. An equivalence relation $E$ is \textbf{concretely classifiable} if $E\leq_B =_{\mathbb{R}}$ where $=_{\mathbb{R}}$ denotes the equality relation on $\mathbb{R}$. We also say $E$ admits \textbf{numerical invariants} if $E$ is concretely classifiable. Poincar\'{e}'s rotation number is an example of numerical invariants. 

Let $S_{\infty}$ be the infinite permutation group. An equivalence relation is \textbf{classifiable by countable structures} if it is Borel reducible to an $S_{\infty}$ action. Since isomorphism of any countable algebraic structures could be regarded as an $S_{\infty}$ action (see \cite{Foremanreduction}), we say an equivalence relation $E$ admits \textbf{complete algebraic invariants} if $E$ is classifiable by countable structures. Topological conjugacy of Cantor systems is classifiable by countable structures (see \cite{CGaoCantorsystem}).

An equivalence relation $E$ on a Polish space $X$ could be regarded as a subset of the product space $X^2$. We say $E$ is \textbf{Borel equivalence relation} if $E$ is a Borel subset of $X^2$. If an equivalence relation is not Borel, then it is impossible to solve it with countably many information. 

A Borel equivalence relation is \textbf{countable} if every equivalent class is countable. Countable Borel equivalence relations are classifiable by countable structures. For symbolic subshifts, the topological conjugacy relation is countable, many results in this direction have been obtained (see \cite{bowen's}, \cite{MarcinToeplitz} and \cite{KayaToeplitz}). 

Many impossibility results regarding classifying dynamical systems have been proved in this direction, especially for isomorphism relation of ergodic measure-preserving transformations (EMPTs). Hjorth \cite{Hjorthturb} proved that the isomorphism relation of EMPTs is not classifiable by countable structures. Foreman and Weiss \cite{Foremanweiss} generalized this result to any generic class of ergodic measure-preserving transformations. A famous result of Foreman, Rudolph and Weiss \cite{NonBorelnessEMPT} proves the non-Borelness of the isomorphism relation of EMPTs. Recently, Foreman and Weiss \cite{SmoothMPTnonBorel} proved isomorphism relation of measure-preserving diffeomorphisms on the torus is not Borel.  Kunde \cite{WeakliymixingnonBorel} generalized this result to weakly mixing diffeomorphisms. Also, Gerber and Kunde \cite{Kakutaninonborel} proved non-Borelness for Kakutani equivalence of measure preserving diffeomorphisms on the torus.

For topological conjugacy, the break through anti-classification result is to Hjorth \cite{Hjorthbook} who proved that the topological conjugacy of homeomorphisms on the unit square is not classifiable by countable structures. Let $M$ be an $n$-dimensional manifold, Foreman and Gorodetski \cite{Foremangorodiff} proved that topological conjugacy relation of smooth diffeomorphisms on $M$ is not concretely classifiable if $n\geq 2$ and not Borel when $n\geq 5$, and the latter was recently generalized by Foreman and Gorodetski and independently Vejnar  to all manifolds.  In \cite{turbulentPeng}, the author proved that when $n\geq 2$, the topological conjugacy of diffeomorphism is not classifiable by countable structures.  However, all those results used fixed point in an essential way. Thus, Smale's classification program regarding topological minimal transformations, as a qualitative version of ergodicity, remains largely open. This question was also asked in \cite[Problem 2]{SmoothMPTnonBorel}.

\begin{question}
    {\rm (Foreman and Weiss) \cite[Problem 2]{SmoothMPTnonBorel}} \textit{ The problem of classifying diffeomorphisms of compact surfaces up to topological conjugacy remains largely open. Work of the first author with A. Gorodetski shows that the isomorphism relation itself is not Borel, but for a very specific type of diffeomorphisms of manifolds of dimension 5 and above. It is not know, for example for topologically minimal transformations.}
\end{question}

Note that Kunde \cite{Kundesmoothconjugacy} proved that the smooth conjugacy relation of orientation preserving $C^{\infty}$-diffeomorphisms on the unit circle is not concretely classifiable. Since all those diffeomorphisms are minimal due to the famous result of Denjoy, this shows smooth conjugacy behaves differently than topological conjugacy.  In \cite{BoLiminimal}, by applying a completely different technique, the author and Li was able to show the isomorphism relation of minimal Cantoriod systems is not classifiable by countable structures. However, Cantorids is a class of compact metric spaces which are very far from being a manifold and it is not clear whether the space can be fixed in the proof of \cite{BoLiminimal}. Based on this result, Vejnar \cite{Vejnar} asks that can we realize the complexity of minimal systems on a concrete topological space. It is worth mentioning that we still do not know whether the isomorphism relation of minimal dynamical systems in general is a complete orbit equivalence relation.  An exciting result of Deka, Kwietniak, Garc\'{i}a-Ramos and Kasprzak and Kunde \cite{DGKKK} shows that the conjugacy relation of Cantor minimal systems is not Borel.

In this paper, we show that the topological conjugacy of minimal homeomorphisms on the torus is not classifiable by countable structures. However, our construction uses $C^0$-diffeomorphisms. Still, the torus is the first known manifold, or even a compact metric space whose minimal homeomorphisms can have such a high complexity to classify.

\begin{theorem}
    The topological conjugacy relation of minimal homeomorphisms on the torus is not classifiable by countable structures.
\end{theorem}

 We use the Approximation by Conjugation method developed by Rees in \cite{Rees}. The exact construction follows an earlier paper of Rees \cite{MaryReesnondistal}. Also, we will use the theory of turbulence developed by Hjorth in descriptive set theory to show the anti-classification. 

\subsection{Outline of the paper}
In Section \ref{construction}, we give the construction of our Borel reduction. In Section \ref{existlimit}, we show that why the systems we constructed are well-defined. Then in Section \ref{tt section}, we define an equivalence relation $E_{\rm tt}$ which is helpful for the proof. Finally, we prove that our construction is a Borel reduction from $c_0$ to the conjugacy of minimal homeomorphisms on the torus in Section \ref{endgame}.


\section{preliminaries}
 In this paper, we use $\T$ to denote the circle and $\T^2$ to denote the torus. Denote by ${\rm Homeo}(\mathbb{T}^2)$ the space of homeomorphisms on torus.
 
 Let $a,b\in \T^2$, we denote by $T_{a,b}$, the rotation map on $\T^2$:
$$
T_{a,b}(x,y)=(x+a,y+b).
$$
 It is well-known that $T_{a,b}$ is minimal if and only if $a,b$ are rationally independent, i.e., there does not exist rational numbers $q_1,q_2,q_3$ which are not all equal to $0$, such that $q_1a+q_2b+q_3=0$.


   We use $||\cdot||$ to denote the usual metric on $\T^2$. For a set $A\subset \T^2$, we use the following notation:  
   $$
  {\rm Diam}(A)= {\rm Sup}\{||x-y||, x,y \in A\}.
   $$
    $||\cdot||_{\infty}$ will be the metric on ${\rm Homeo}(\T^2)$, defined as:

    $$
    ||f-g||_{\infty}={\rm Sup}_{x\in \T^2}||f(x)-g(x)||.
    $$

Let $G$ be a Polish group acting continuously on a Polish space $X$, we say the action is \textbf{turbulent}, if

\begin{itemize}
    \item Every orbit is dense.
    \item Every orbit is meager.
    \item Every local $U$-$V$-orbit is somewhere dense.
\end{itemize}
    Local orbits are defined as follows. Let $x\in X$, $U$ be an open set in $X$ and $V$ be an open neighborhood of the identity in $G$. Define the \textbf{local $U$-$V$-orbit} $\mathcal{O}(x,U,V)$ as
    $$
    \{y\in U\mid \forall i\in \{1,2\dots,n\}\,\,g_i\in V\,\,y=g_1\circ g_2\dots g_n x\}.
    $$
    
   Hjorth's turbulence theorem \cite{Hjorthbook} states that all turbulent actions are not Borel reducible to any Borel $S_{\infty}$ action. Donote by $c_0$ the following equivalence relation defined on $\mathbb{R}^{\omega}$, for $\alpha,\beta\in \mathbb{R}^{\omega}$,
  $$
   \alpha\,c_0\,\beta \,\,\mbox{if and only if}\,\, \lim_{n}|\alpha(n)-\beta(n)|=0.
   $$
    It is well-known \cite[Theorem 6.2.2]{Kanoveibook} that $c_0$ is actually a turbulent group action and it is Borel bireducible with its restriction to $[0,1]^{\omega}$.

   We will show that $[0,1]^{\omega}/c_0$ is Borel reducible to topological conjugacy relation of the minimal homeomorphisms on the torus.
   
\subsection{How to avoid using fixed points?} Let $f\in {\rm Homeo}(X)$, two points $x,y\in X$ are \textbf{asymptotic} in $(X,f)$ if ${\rm lim}_{n}(f^nx,f^ny)=0$ as $n \rightarrow \infty$. It is routine to check that this is an equivalence relation which is preserved under conjugacy. In other words, if $h\in {\rm Homeo}(X)$ and $g=hfh^{-1}$, then $x,y$ are asymptotic in $(X,f)$ if and only if $h(x)$ and $h(y)$ are asymptotic in $(Y,g)$. By \textbf{asymptotic class}, we mean the equivalence class of the asymptotic equivalence relation.

We want to send an element $\alpha\in [0,1]^{\omega}$ to a minimal homeomorphism $f_{\alpha}$ on $\T^2$ such that this is a Borel reduction.

By adding conditions to the Approximation by Conjugation method, we end up with a minimal homeomorphism $f_{\alpha}$ with the property that for any $x\in \T^2$, the number of elements in its asymptotic class is either one or continuum. Since asymptotic classes are preserved under conjugacy, if two homeomorphisms are conjugate, the points with asymptotic class of size continuum will be mapped to points with asymptotic class of size continuum. Then we can let those points play the same role as fixed points played in Hjorth's proof \cite[Theorem 4.17]{Hjorthbook} or the author's proofs in \cite{turbulentPeng}.

\section{The construction of the Borel reduction} \label{construction}
 In this section, we follow notation of Mary Rees \cite{MaryReesnondistal}. First, we fix a minimal rotation of $\T^2$, denoted by $T$,  
 
    $$
    T(x,y)=(x+a,y+b).
    $$
    
    Through the whole paper, this $T$ will be fixed. We will also fix an orbit $(T^n(z_0))_{n\in \mathbb{Z}}$ for some $z_0\in \T^2$. For $n\geq 0$, let 
    $$
    z_{2n}=(x_{2n},y_{2n})=T^{n}z_0
    $$ 
    and 
    $$
      z_{2n+1}=(x_{2n+1},y_{2n+1})=T^{-n-1}z_0.
    $$
    See the following picture:

    
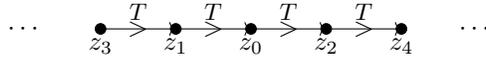
\begin{center}
    \begin{tikzpicture}
        \filldraw (0,0) circle (2pt);
        \filldraw (1,0) circle (2pt);
        \filldraw (2,0) circle (2pt);
        \filldraw (3,0) circle (2pt);
        \filldraw (4,0) circle (2pt);

        \node at (-1,0) {$\cdots$};

        \node at (5,0) {$\cdots$};
        
        \node[below] at (0,0) {$z_3$};
        \node[below] at (1,0) {$z_1$};
        \node[below] at (2,0) {$z_0$};
        \node[below] at (3,0) {$z_2$};
        \node[below] at (4,0) {$z_4$};
     \draw[->] (0,0) -- (1,0) node[midway, above] {\small $T$} node[midway] {$>$};
        \draw[->] (1,0) -- (2,0) node[midway, above] {\small $T$} node[midway] {$>$};
        \draw[->] (2,0) -- (3,0) node[midway, above] {\small $T$} node[midway] {$>$};
        \draw[->] (3,0) -- (4,0) node[midway, above] {\small $T$} node[midway] {$>$};
    \end{tikzpicture}
\end{center}

\maketitle

    Now we define two bijections $\sigma$ and $\tau$ from $\mathbb{N}$ to $\mathbb{N}$ such that $Tz_n=z_{\sigma(n)}$ and $T^{-1}z_n=z_{\tau(n)}$. Of course,  we have $0<|\sigma(n)-n|\leq 2$ and $0<|\tau(n)-n|\leq 2$.


    
    For any $\alpha\in [0,1]^{\omega}$, we will define $S_{\alpha}\in {\rm{Homeo}(\T^2)}$ and a continuous surjection $\Phi_{\alpha}: \T^2 \rightarrow  \T^2$ such that 
    $$
          \Phi_\alpha \circ S_\alpha = T\circ \Phi_\alpha.
    $$

    Instead of directly defining $S_{\alpha}$, we will define $\Phi_{\alpha}$ as a self-homeomorphism on a $T$-invariant comeager set $X=(\T-\{x_n\}_n)\times \T$ of $\T^2$, then we extend $\Phi^{-1}_{\alpha}T\Phi_{\alpha}$ to the whole $\T^2$.  In order to make the extension well-defined, we want $\Phi_{\alpha}$ to satisfy the following conditions:

    \begin{itemize}
        \item [(A)]$\Phi_{\alpha}$ is a continuous onto map of the form $\Phi_{\alpha}(x,y)=(x,\phi_{\alpha}(x,y))$ where $\phi_{\alpha}$ is a continuous map from $\T^2$ to $\T$.
        \item [(B)]For any $z\in \T^2$, we have $\Phi_{\alpha}^{-1}(\{z\})$ is not a singleton if and only if $z=z_n$ for some $n\in \mathbb{N}$ and 
        $$
          \Phi_{\alpha}^{-1}(\{z_n\})=\{x_n\}\times I_n
        $$
        where $I_n$ is an interval.
         
       \item [(C)] $\Phi_{\alpha}^{-1}\circ T \circ\Phi_{\alpha}$ and  $\Phi_{\alpha}^{-1}\circ T^{-1}\circ\Phi_{\alpha}$ are uniformly continuous on $X$.
    \end{itemize}

    The condition (C) guarantees that $S_{\alpha}$ could be a well-defined self-homeomorphism of the torus. Indeed, if we can find such a continuous function $\Phi_{\alpha}$, we can take $S_{\alpha}$ to be the unique extension of $\Phi_{\alpha}^{-1}\circ T \circ\Phi_{\alpha}$. 

    \begin{lemma}\label{R1 works}
        The system $(\T^2,S_{\alpha})$ is minimal.
    \end{lemma}
    \begin{proof}
       We will argue by contradiction. Let $(Y,S_{\alpha})$ be an invariant subsystem of $(\T^2,S_{\alpha})$. Since the rotation $(\T^2,T)$ is minimal and $\Phi_{\alpha}$ is a factor map from $(\T^2,S_\alpha)$ to $(\T^2,T)$, we know that $\Phi_{\alpha}(Y)=\T^2$. But by condition (2), we have that $\Phi_{\alpha}^{-1}(\{y\})$ is a singleton, except for $y=z_n$, $n\in \mathbb{N}$. Take $y$ to be any point which is not on the $T$-orbit of $z_0$ such that $\Phi^{-1}_{\alpha}(\{y\})\in Y$. Then we have $\{T^n(y)\}_{n\in \mathbb{Z}}$ is dense in $\T^2$ but $\Phi^{-1}_{\alpha}(\{T^ny\}_{n\in \mathbb{Z}})\subset Y$ is not dense. This contradicts the fact that $\Phi_{\alpha}$ is a surjection.
    \end{proof}

   For each $\alpha\in \mathbb{R}^{\omega}$,  we need to construct $\Phi_{\alpha}$ satisfying conditions (A)-(C). $\Phi_{\alpha}$ will be defined as a uniform limit of $\Phi_{\alpha,n}$ for a sequence of homeomorphisms $\Phi_{\alpha,n}\in {\rm Homeo}(\T^2)$. Each $\Phi_{\alpha,n}$ will have the form $\Phi_{\alpha,n}=\Lambda_{\alpha,n}\circ \Lambda_{\alpha,n-1}\circ \cdots \circ \Lambda_{\alpha,1}$ where for each $k\leq n$, the function $\Lambda_{\alpha,k}$ is a continuous function from $\T^2$ to $\T^2$ such that, 
    \begin{enumerate}
        \item $\Lambda_{\alpha,n}$ is the identity out of $A_n$, where $A_n$ is of the form $\{(x_n+\delta,y_n+\delta)|\,\, -\delta_{\alpha,n}\leq \delta\leq \delta_{\alpha,n}\}$. In other words, $A_n$ is a square centered at $z_n$ with parameter $2\delta_{\alpha,n}$. 

        \begin{center}
\begin{tikzpicture}
    \draw[thick] (0,0) rectangle (7,7);
    
    \foreach \x/\y/\size/\label in {
        3/3/1.8/$z_1$,      
        5/5/1.5/$z_2$,
        1/5/1.2/$z_3$,
        5/1/1.0/$z_4$,
        1/1/0.5/$z_5$,
        5/3/0.4/$z_6$,
        3.2/2.5/0.2/$z_7$       
    } {
        \pgfmathsetmacro{\half}{\size/2}
        \draw[thick] (\x - \half, \y - \half) rectangle (\x + \half, \y + \half);
        \filldraw[black] (\x,\y) circle (1.5pt);
        \node[above right] at (\x, \y) {\label};
    }
   
\end{tikzpicture}

   Construction of the squares $A_{\alpha,n}$
\end{center}
        \item  $\Lambda_n(A_n)=A_n$ and $\Lambda_{\alpha,n}$ is of the form $\Lambda_{\alpha,n}(x,y)=(x,\lambda_{\alpha,n}(x,y))$ where $\lambda_{\alpha,n}$ is a continuous function from $\T^2$ to $\T$. Also, $(\Lambda_{\alpha,n})^{-1}(\{(x,y)\})$ is singleton except when $(x,y)=z_n$ and 
        $$
         (\Lambda_{\alpha,n})^{-1}(z_n)=\{x_n\}\times [y_n-\frac{\delta_{\alpha,n}}{2},y_n+\frac{\delta_{\alpha,n}}{2}].
        $$
        \item $(\Lambda_{\alpha,\sigma(n)})^{-1}\circ T\circ \Lambda_{\alpha,n}$ and $(\Lambda_{\alpha,\tau(n)})^{-1}\circ T^{-1}\circ \Lambda_{\alpha,n}$ are uniformly continuous on $\T^2-\{z_m\}_{m\in \mathbb{N}}$.
        \item When $|n-m|\leq 5$, the squares $A_m$ and $A_n$ are disjoint, hence, $\Lambda_{\alpha,n}$ and $\Lambda_{\alpha,m}$ are commute.
        \item $\delta_{\alpha,n}<1/2^n$ and is small enough so that 
        \begin{itemize}
        
            \item  $\{x_m\} \times \T $ is disjoint with  $A_{n}$ for all $m<n$.
            \item for every $n$ and for all $m<n$
        $$
         \delta_{\alpha,n}<\frac{||x_m-x_n||}{n}.
        $$
            \item for $\Phi_{\alpha,n-1}$ and $x,y\in \{x_n\}\times \T$:
        $$
        ||\Phi_{\alpha,n-1}(x),\Phi_{\alpha,n-1}(y)||\leq \delta_{\alpha,n} \Rightarrow ||x-y||<1/n.
        $$
        \end{itemize}
 
        \item There exists a sequence $\{a_n\}_{n\in \mathbb{N}}$ of integers $\geq2$ with $\prod^{\infty}_{n=1}(1-1/a_n)>0$ and if $x,y\in \T^2$ with $|x-y|\geq1/n$, then for $m>n$, we have
        $$
         ||\Phi_{\alpha,m}(x)-\Phi_{\alpha,m}(y)||\geq \Pi_{r=n+1}^{m}(1-1/a_r)||\Phi_{\alpha,n}(x)-\Phi_{\alpha,n}(y)||.
        $$
        \item Let $\Phi^{\sigma}_{\alpha,n}=\Lambda_{\alpha,\sigma(n)}\circ\cdots\circ\Lambda_{\alpha,\sigma(1)}$ and $\Phi^{\tau}_{\alpha,n}=\Lambda_{\alpha,\tau(n)}\circ\cdots\circ \Lambda_{\alpha,\tau(1)}$. And the sequence of functions $\{(\Phi^{\sigma}_{\alpha,n})^{-1}\circ T\circ\Phi_{\alpha,n}\}_n$ and $\{(\Phi^{\tau}_{\alpha,n})^{-1}\circ T^{-1}\circ\Phi_{\alpha,n}\}_n$ are uniformly continuous on $X$ and those two sequences of functions are uniformly Cauchy on $X$.

        

        \end{enumerate}
         
         Note that 
         $$
         ||\Phi_{\alpha,n}-\Phi_{\alpha,n-1}||_{\infty}=||\Lambda_{\alpha,n}-id||_{\infty}<\delta_{\alpha,n}.$$
         
         So $\{\Phi_{\alpha,n}\}$ converges to a continuous map by (5), denote this map by  $\Phi_{\alpha}$. 

      Since we list out a lot of conditions, we give a brief explanation on how those conditions work. Conditions (5) is a numerical condition which will be used to justify that our function is a Borel reduction.  The rest of the conditions are the same as conditions in Rees's paper \cite{MaryReesnondistal}, the purpose of having those conditions is to make sure that the limit $\Phi_{\alpha}={\rm lim}_n \Phi_{\alpha,n}$ exists and $S_{\alpha}$ is well-defined. 

       To define $\Phi_{\alpha,n}$, we need 
  to first define a continuous function $h^{\alpha}:[-1,1]^2\rightarrow [-1,1]$ where we write $h^{\alpha}_t(s)$ for $h^{\alpha}(t,s)$. We first list out some properties $h^{\alpha}$ will have:

   \begin{itemize}
       \item[(a)]$h^{\alpha}_t$ is a homeomorphism for $t\neq 0$.
       \item[(b)] $h^{\alpha}_1=h^{\alpha}_{-1}={\rm Id}$, $h^{\alpha}_t(1)=1$, $h^{\alpha}_t(-1)=-1$ for all $t$.
       \item[(c)] $h^{\alpha}_t([-1/2,1/2])=[-|t|/2,|t|/2]$ for all $t\neq 0$. And
       $$
       (h^{\alpha}_0)^{-1}(\{0\})=[-1/2,1/2].
       $$
       \item[(d)] $h_0$ is onto and $(h^{\alpha}_0)^{-1}(\{t\})$ is a singleton for all $t\neq0$.
       \item[(e)] $h^{\alpha}_t=h^{\alpha}_0$ on $(h^{\alpha}_0)^{-1}([-1,-t]\cup[t,1])$.
    \end{itemize}


 Now we define $\Lambda_{\alpha,n}$. In the square $A_{\alpha,n}$, we can write all points in $A_{\alpha,n}$ as $(z+t,z+s)$ where $-\delta_n\leq t,s \leq \delta_n$. We define 
 $$
 \Lambda_{\alpha,n}(z_n+t,z_n+s)=(z_n+t,z_n+h^{\alpha}_t(h^{\alpha}_{\delta_{\alpha,n}})^{-1}(s)).
 $$ 

By the construction above, the function $h^{\alpha}$ completely determines $\Lambda_{\alpha,n}$, $\Phi_{\alpha}$ thus $S_{\alpha}$. All is left is to construct $h^{\alpha}$.

\subsection{Construction of $h^{\alpha}$}

 First we fix a sequence of decreasing numbers $0<\epsilon_n\leq 1$ converging to $0$ with $\epsilon_0=1$.

    We stipulate $h^\alpha$ to have the property that $h^{\alpha}_t=h^{\alpha}_{-t}$. In this way, we only need to define $h^{\alpha}(t,s)$ for all $t\geq0$.

   Let $H_n=[\epsilon_n,-\epsilon_{n+1})\times [-1,1]$.
   
       For each $k\in \mathbb{N}$, we divide the rectangle $H_{k}$ into four parts $H^1_{k},H^2_{k},H^3_{k}$ and $H^4_{k}$ where
       $$
       H^i_{k}=[d^i_{k},e^i_{k}]\times[-1,1]\,\,i\in\{1,2,3,4\}
       $$
      where
      $$
        \epsilon_{k+1}=d^1_{k}<e^1_{k}=d^2_{k}<\cdots<e^4_{k}=\epsilon_{k}.
        $$

   \begin{center}

\begin{tikzpicture}
    \draw[thick] (0, 0) rectangle (2, 4);

    \coordinate (A) at (0, 4);
    \coordinate (B) at (0.5, 4);
    \coordinate (C) at (1.0, 4);
    \coordinate (D) at (2, 4);
    \coordinate (E) at (1.5,4);

    \coordinate (P) at (0, 0);
    \coordinate (Q) at (0.5, 0);
    \coordinate (R) at (1, 0);
    \coordinate (S) at (2, 0);
    \coordinate (T) at (1.5,0);

    \draw[thick] (B) -- (Q);
    \draw[thick] (C) -- (R);
    \draw[thick] (A) -- (P);
    \draw[thick] (D) -- (S);
    \draw[thick] (E) -- (T);

    \filldraw (A) circle (2pt);
    \filldraw (B) circle (2pt);
    \filldraw (C) circle (2pt);
    \filldraw (D) circle (2pt);
    \filldraw (E) circle (2pt);
    \filldraw (P) circle (2pt);
    \filldraw (Q) circle (2pt);
    \filldraw (R) circle (2pt);
    \filldraw (S) circle (2pt);
    \filldraw (T) circle (2pt);
   
  
    \node at (S) [below right] {$\epsilon_{k}$};
    \node at (0,0) [below left] {$\epsilon_{k+1}$};
   
\end{tikzpicture}

The division
\end{center}

    All those partitions are fixed from now on. Before we define exactly what $h_{\alpha}$ is, we make some comment on the whole structure. 

 $\alpha(k)$ will determine the function $h^{\alpha}$ defined on $H_k$. And we will make sure that for any $\beta \in [0,1]^{\omega}$ with the property that $\alpha c_0 \beta$ and for any $\epsilon>0$, one can find $\delta>0$ such that for any $t\in [-\delta,0)\cup(0,\delta]$, 
    $$
     ||(h^\alpha_t)^{-1}-(h^{\beta}_t)^{-1}||_{\infty}<\epsilon.
    $$
     First, we define $h^{\alpha}_0$ to a polyline connecting 
     $$
    (-1,-1),(-1/2,0),(1/2,0),(1,1).
     $$
         

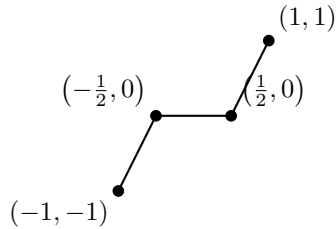
\begin{figure}[h!]
    \centering
    \begin{tikzpicture}
        \draw[thick] (-1, -1) -- (-1/2, 0);
        \draw[thick] (-1/2,0) -- (1/2,0);
        \draw[thick] (1/2,0) -- (1,1);
        \filldraw (-1, -1) circle (2pt);
        \node at (-1, -1) [below left] {$\left( -1, -1 \right)$};

        \filldraw (-1/2, 0) circle (2pt);
        \node at (-1/2, 0) [above left] {$\left( -\frac{1}{2}, 0 \right)$};

        \filldraw (1/2, 0) circle (2pt);
        \node at (1/2, 0) [above right] {$\left( \frac{1}{2}, 0 \right)$};

        \filldraw (1, 1) circle (2pt);
        \node at (1, 1) [above right] {$\left( 1, 1 \right)$};
    \end{tikzpicture}
    \caption{The function $h^\alpha_0$}
\end{figure}

     Note that we have $h^{\alpha}_0=h^{\beta}_0$ for all $\beta\in [0,1]^{\omega}$. For all $t\in (0,1]$, our function $h_{t}^{\alpha}$ will be a polyline connecting 
    $$
     (-1,-1),((-1-t)/2,-t),(-1/2,-t/2),(-p_{\alpha}(t),t/4),(1/2.t/2),((1+t)/2,t),(1,1)
    $$
    where $p_{\alpha}(t)$ is a continuous function.
    
    \begin{center} 
    \begin{tikzpicture}
        \draw[thick] (-2,-3) -- (-1,-1) --(-0.5,-0.5)-- (-0.3,0.2) -- (0.5,0.5)--(1,1)--(2,3);
        
        \fill[black] (-2,-3) circle (2pt);
        \fill[black] (-1,-1) circle (2pt);
        \fill[black] (-0.5,-0.5) circle (2pt);
        \fill[black] (-0.3,0.2) circle (2pt);
        \fill[black] (0.5,0.5) circle (2pt);
        \fill[black]  (1,1) circle (2pt);
        \fill[black]  (2,3) circle (2pt);
        
           \node[left] at (-2,-3) {$(-1,-1)$};
           \node[left] at (-1,-1) {$(\frac{-t-1}{2},-t)$};
           \node[left] at (-0.3,0.2) {$(-p_{\alpha}(t),t/4)$};
           \node[right] at (-0.5,-0.5)  {$(-\frac{1}{2},-\frac{t}{2})$};
           \node[right]  at (0.5,0.5){$(\frac{1}{2},\frac{t}{2})$};
           \node[right]  at (1,1) {$(\frac{1+t}{2},t)$};
           \node[right]  at (2,3)  {$(1,1)$};

        \node[below] at (0,-3) {The function $h_t^{\alpha}$};
    \end{tikzpicture}
\end{center}

    Now we define $h_\alpha$ on $H_k$ since $h_\alpha$ is symmetric, the other part is automatically defined. From now on, we will fix a dense subset $\{q_n\}_{n\in \mathbb{N}}$ of $[0,1]$.

    Note that the graph of the function $h^{\alpha}$ is a polyline, thus, the point $(-p_{\alpha}(t),t/4)$ determines the function. If we define $p_{\alpha}(t)$ in a continuous way with respect to $t$ in $H_k$, then the function defined on $H_k$ will be continuous.

    On $H_1^0$, since we know $h_{1}^{\alpha}$ must be identity, we first set $p_{\alpha}(t)$ to be $\alpha(0)/4$ for $t\in[d^3_0,e^3_0]$, and let $\mathcal{V}_0$ to be the orientation preserving affine homeomorphism from $[1,e^3_0]$ to $[1/4,\alpha(0)/4]$ and set $p_{\alpha}(t)=\mathcal{V}_0(t)$ for all $t\in [1,e^3_0]$.

    Now set $p_{\alpha}(t)=q_0/4$ for $t\in [d^1_0,e^1_0]$ and let $\mathcal{W}_0$ be the affine  homeomorphism from $[e^2_0,e^1_0]$ to $[\alpha(0)/4, q_0/4]$. Take $p_{\alpha}(t)$ to be $\mathcal{W}_0(t)$ for all $t\in [e^2_0,e^1_0]$.

    In the $k^{\rm th}$ step, first, we define $p_{\alpha}(t)$ for $t\in [d^1_k,e^1_k]$ and $[d^3_k,e^3_k]$. 

    For $k\geq 1$, let $p_{\alpha}(t)=\alpha(k)/4$ for $t\in [d^1_k,e^1_k]$ and $p_{\alpha}(t)=q_k/4$ for $t\in [d^3_k,e^3_k]$. Now let $\mathcal{V}_k$ be the orientation preserving affine homeomorphism from $[\epsilon_k, e^3_k]$ to $[q_{k-1}/4,\alpha(k)/4]$. And let $\mathcal{W}_k$ be the orientation preserving affine homeomorphism from $[d^2_k,e^2_k]$ to $[\alpha(k)/4,q_k/4]$. Let $p_{\alpha}(t)=\mathcal{V}_k(t)$ for $t\in [\epsilon_k, e^3_k]$ and $p_{\alpha}(t)=\mathcal{W}_k(t)$ for $t\in [d^2_k,e^2_k]$.

    It is routine to check that $h^{\alpha}$ satisfies all conditions (1)-(5).

\section{Existence of the limit} \label{existlimit}
  In this section, we show that the limit $\Phi_{\alpha}$ exists for all $\alpha\in [0,1]^{\omega}$. Note that this proof is extracted from \cite{MaryReesnondistal} to make this paper self-contained. This technique is also called the \textbf{Denjoy-Rees} technique. 

In this section, we fix an $\alpha\in [0,1]^{\omega}$. Clearly, there exists a sequence of decreasing positive reals $\{\epsilon_{1,n}\}_n$ such that conditions (4),(5) are satisfied when $\delta_{\alpha,n}<\epsilon_{1,n}$.

  \begin{lemma} \label{uniformcont}
    Fix $m\in \mathbb{N}$. Let $R:\T^2 \rightarrow \T^2$ be defined as 
      $$
       R(z_m+x)=z_n+x\,\,\,x\in \T^2.
      $$
      then $\Lambda_{\alpha,n}^{-1}\circ R \circ \Lambda_{\alpha,m}$ is uniformly continuous on $\T^2\setminus\{z_m\}$.
  \end{lemma}

  \begin{proof}

      It is enough to show that our function is uniformly continuous on $A_{\alpha,m}\cap R^{-1}(A_{\alpha,n})\setminus \{z_m\}$. Indeed, on the complement of the interior of $ A_{\alpha,m}\cap R^{-1}(A_{\alpha,n})\setminus \{z_m\}$, the function $\Lambda_{\alpha,n}^{-1}\circ R \circ \Lambda_{\alpha,m}$ is a well-defined continuous function defined on a compact set, hence uniformly continuous.

      We may assume that $z_m+x$ is in $A_{\alpha,m}\cap R^{-1}(A_{\alpha,n})\setminus \{z_m\}$. Let $x=(x_1,x_2)$, by the definition of $\Lambda_{\alpha,n}$, we have
      $$
   \Lambda_{\alpha,n}^{-1}\circ R \circ \Lambda_{\alpha,m}(z_m+x)=\Lambda_{\alpha,n}^{-1}(z_{n}+(x_1,h^{\alpha}_{x_1}\circ (h^{\alpha}_{\delta_{\alpha,m}})^{-1}(x_2))=
      $$
      $$
       z_n+(x_1,h^{\alpha}_{\delta_{\alpha,n}}\circ h^{\alpha}_{x_1}\circ (h^{\alpha}_{x_1})^{-1} \circ (h^{\alpha}_{\delta_{\alpha,m}})^{-1}(x_2))=z_n+(x_1,h^{\alpha}_{\delta_{\alpha,n}}\circ (h^{\alpha}_{\delta_{\alpha,m}})^{-1}(x_2)).
      $$
      Since the function $h^{\alpha}_{\delta_{\alpha,n}}\circ (h^{\alpha}_{\delta_{\alpha,m}})^{-1}$ is uniformly continuous function, the lemma follows.
  \end{proof}

  \begin{theorem} \label{welldefinedness}
      We can find a sequence of triples $(\epsilon_{1,n},\epsilon_{2,n},\epsilon_{3,n})$ such that when $\D(A_{\alpha,n})<\epsilon_{i,n}$ for all $i=1,2,3$,  we have that conditions (6),(7) will be satisfied for $\Phi_{\alpha}$.
  \end{theorem}

  \begin{proof}
        Note that we have chosen a sequence of reals $\{\epsilon_{1,n}\}$ in the beginning of this section. We are proceeding to find $\{\epsilon_{2,n}\}_n$ and $\{\epsilon_{3,n}\}_n$.
        
        We first define the sequence $\{\epsilon_{2,n}\}_n$. First, we fix a sequence of integers $\{a_n\}_n$ such that $\prod_{n=1}^{\infty}(1-\frac{1}{a_n})>0$. For any $t>0$ and a square $A$, we use the notation $tA$ to denote the square with the same center as $A$, and sides $t$-times as long as those of $A$. We take $\epsilon_{2,0}=\epsilon_{1,0}$. For any $n\geq 1$, we take $\epsilon_{2,n}$ to be so small such that when $\delta_{\alpha,n}<\epsilon_{2,n}$
        $$
        \D(\Phi^{-1}_{\alpha,n-1}((a_n+1)A_{\alpha,n}))<\frac{1}{n}.
        $$
        Since $A_n$ is disjoint with all $\{x_m\}\times \T$ for $m<n$, we have that $\Phi_{\alpha,n-1}$ is a homeomorphism restricted on $A_{\alpha,n}$.

        Now we define the sequence $\{\epsilon_{3,n}\}_n$. First, we let $\epsilon_{3,0}=\epsilon_{1,0}$. By conditions (4) and (5), we can choose a sequence of closed subsets $\{N_s\}_{s\in \mathbb{N}}$ such that for all $s\geq 0$, we have 
        \begin{itemize}
            \item [(i)]$z_i\not\in N_s$ for all $i\leq s$.
            \item [(ii)]For $s+1\leq i\leq s+5$, we have $A_i\subset N_s$ and $T(A_{{s+3}})$, $T^{-1}(A_{s+3})$ are subsets of $N_s$. Recall that $T$ is the minimal rotation we fixed in the beginning of the construction. For example, one can take $N_s$ to be the shaded region in the figure below.

            \begin{center}
    
\begin{tikzpicture}

\draw[ultra thick, black] (0,0) rectangle (12,12);

\draw[thick, black] (1,1) rectangle (5,5);      
\draw[thick, black, fill=black!30] (6,1) rectangle (9,4);      
\draw[thick, black, fill=black!30] (10,1) rectangle (11.5,2.5); 
\draw[thick, black, fill=black!30] (1,6) rectangle (4,9);      
\draw[thick, black, fill=black!30] (5,6) rectangle (7,8);      
\draw[thick, black, fill=black!30] (8,6) rectangle (9.5,7.5);  
\draw[thick, black, fill=black!30] (1,10) rectangle (2.5,11.5); 
\draw[thick, black, fill=black!30] (3.5,10) rectangle (4.5,11); 

\node at (3,3) [below left] {$A_{\alpha,s}$};

\end{tikzpicture}

Picture of $N_s$
\end{center}
            
            \item [(iii)]By the previous condition, we know that $\Phi_{\alpha,s}$ is a homeomorphism when restricted on $N_s$. So we can find $\eta_s>0$ such that for all $x,y\in N_s$,
            $$
            ||x-y||\leq \eta_s \Rightarrow\,\,||\Phi_{\alpha,s}^{-1}(x)-\Phi_{\alpha,s}^{-1}(y)||<\frac{1}{2^s}.
            $$
        \end{itemize}
        We take $\epsilon_{3,n}$ to be ${\rm Min}(\eta_i:n-1\leq i\leq n-5)$.

        We have defined the tuple $(\epsilon_{1,n},\epsilon_{2,n},\epsilon_{3,n})$, now we start to verify that it satisfies the assertion of Theorem \ref{welldefinedness}.

        \begin{lemma} \label{NC(6)}
            If $\D(A_{\alpha,n})<\epsilon_{2,n}$, then condition (6) will be satisfied for $\{\Phi_{\alpha,n}\}_n$.
        \end{lemma}
        \begin{proof}
            Let $||x-y||\geq \frac{1}{n}$. For any $m> n$, by the definition of $\{\epsilon_{2,n}\}_n$, we know that 
            $$
            \D(\Phi_{\alpha,m}^{-1}((a_{m+1}+1)A_{\alpha,m+1}))<\frac{1}{m+1}<\frac{1}{n}.
            $$ Thus, the points $\Phi_{\alpha,m}(x)$ and $\Phi_{\alpha,m}(y)$ are not both in $(a_{m+1}+1)A_{\alpha,m+1}$.
            
            \textbf{Case 1.} $\Phi_{\alpha,m}(x)$ and $\Phi_{\alpha,m}(y)$ are both not  in $(a_{m+1}+1)A_{\alpha,m+1}$.
            
               Since $\Lambda_{\alpha,m+1}$ is equal to the identity outside $A_{\alpha,m+1}$, we know that 
               $$
               ||\Phi_{\alpha,m+1}(x)-\Phi_{\alpha,m+1}(y)||=||\Phi_{\alpha,m}(x)-\Phi_{\alpha,m}(y)||>(1-\frac{1}{a_m})||\Phi_{\alpha,m}(x)-\Phi_{\alpha,m}(x)||.
               $$
               This implies condition (6).
               
            \textbf{Case 2.} Exactly one of $\Phi_{\alpha,m}(x)$ and $\Phi_{\alpha,m}(y)$ is in $(a_{m+1}+1)A_{\alpha,m+1}$. 
            
              Note that if  $\Phi_{\alpha,m}(x)\in (1+a_{m+1})A_{\alpha, m+1}\setminus A_{\alpha,m+1}$, then since $\Lambda_{\alpha,m+1}$ is equal to the identity outside $A_{\alpha,m+1}$, the same argument as in Case 1 works. Hence, we may assume that $\Phi_{\alpha,m}(x)\in A_{\alpha, m+1}$ and $\Phi_{\alpha,m}(y)\not\in (a_{m+1}+1)A_{\alpha,m+1}$.  This implies that 
              $$
              ||\Phi_{\alpha,m}(x)-\Phi_{\alpha,m}(y)||\geq a_{m+1}\D(A_{\alpha,m+1}).
              $$
              We have
              $$
              ||\Phi_{\alpha,m+1}(x)-\Phi_{\alpha,m}(x)||=||\Lambda_{\alpha,m+1}\circ \Phi_{\alpha,m}(x)-\Phi_{\alpha,m}(x)||<\D(A_{\alpha,m+1}).
              $$
              So,
              $$
               ||\Phi_{\alpha,m+1}(x)-\Phi_{\alpha,m+1}(y)||=||\Phi_{\alpha,m+1}(x)-\Phi_{\alpha,m}(y)|| \geq
               $$
               $$
               \geq ||\Phi_{\alpha,m}(x)-\Phi_{\alpha,m}(y)||-||\Phi_{\alpha,m+1}(x)-\Phi_{\alpha,m}(x)||>||\Phi_{\alpha,m}(x)-\Phi_{\alpha,m}(y)||-\D(A_{\alpha,m+1}) \geq
               $$
               $$
               \geq (1-\frac{1}{a_{m+1}})||\Phi_{\alpha,m}(x)-\Phi_{\alpha,m}(y)||.
              $$
              This implies condition (6).
        \end{proof}
        Note that $X=(\T-\{x_n\}_n)\times \T$ is a comeager subset of $\T^2$.
        \begin{lemma} \label{controlnumber}
            If $\D(A_{\alpha,n})<\epsilon_{1,n}$, then we have
            $$
            \SU_{x\in X}||(\Phi^{\sigma}_{\alpha,n})^{-1}\circ T \circ \Phi_{\alpha,n}(x)-(\Phi^{\sigma}_{\alpha,n+1})^{-1}\circ T \circ \Phi_{\alpha,n+1}(x)||\leq
            $$
            $$
            \leq {\rm Max}_{n-1\leq r\leq n+3}\{\SU_{x\in X}||\Phi_{\alpha,n-2}^{-1}\circ \Lambda_{\alpha,r}(x)-\Phi_{\alpha.n-2}^{-1}(x)||\}+\SU_{x\in X}||\Phi^{-1}_{\alpha,n-2}\circ T\circ \Lambda_{\alpha,n+1}(x)-\Phi_{\alpha,n-2}^{-1}\circ T(x)||.
            $$
            and
            $$
            \SU_{x\in X}||(\Phi^{\tau}_{\alpha,n})^{-1}\circ T^{-1} \circ \Phi_{\alpha,n}(x)-(\Phi^{\tau}_{\alpha,n+1})^{-1}\circ T^{-1} \circ \Phi_{\alpha,n+1}|(x)|\leq 
            $$
            $$
            \leq{\rm Max}_{n-1\leq r\leq n+3}\{\SU_{x\in X}||\Phi_{\alpha,n-2}^{-1}\circ \Lambda_{\alpha,r}(x)-\Phi_{\alpha.n-2}^{-1}(x)||\}+\SU_{x\in X}||\Phi^{-1}_{\alpha,n-2}\circ T\circ \Lambda_{\alpha,n+1}(x)-\Phi_{\alpha,n-2}^{-1}\circ T^{-1}(x)||.
            $$
        \end{lemma}
        \begin{proof}
            Note that $A_{\alpha,n}$ is disjoint with $A_{\alpha,m}$ for all $|m-n|\leq 5$. And the functions $(\Phi^{\sigma}_{\alpha,n})^{-1}\circ T \circ \Phi_{\alpha,n}$ and $(\Phi^{\sigma}_{\alpha,n+1})^{-1}\circ T \circ \Phi_{\alpha,n+1}$ are different only on $A_{\alpha,n+1}$ and $A_{\alpha,\sigma(n+1)}$. On $A=X\cap (A_{\alpha,n+1}\cup A_{\alpha,\sigma(n+1)})$,  we have
            $$
         (\Phi^{\sigma}_{\alpha,n})^{-1}\circ T \circ \Phi_{\alpha,n}= \Phi_{\alpha,n-2}^{-1}\circ T\circ \Phi_{\alpha,n-2}. 
            $$
            and 
            $$
             (\Phi^{\sigma}_{\alpha,n+1})^{-1}\circ T \circ \Phi_{\alpha,n+1}=\Phi_{\alpha,n-2}^{-1}\circ \Lambda_{\alpha,\sigma(n+1)}^{-1} \circ T \circ\Lambda_{\alpha,n+1}\circ \Phi_{\alpha,n-2}.
            $$
            Thus,
            $$ 
          \SU_{x\in A}||(\Phi^{\sigma}_{\alpha,n})^{-1}\circ T \circ \Phi_{\alpha,n}(x)-(\Phi^{\sigma}_{\alpha,n+1})^{-1}\circ T \circ \Phi_{\alpha,n+1}(x)||=
            $$
            $$
            \SU_{x\in A}||\Phi_{\alpha,n-2}^{-1}\circ T\circ \Phi_{\alpha,n-2}(x)-\Phi_{\alpha,n-2}^{-1}\circ \Lambda_{\alpha,\sigma(n+1)}^{-1} \circ T \circ\Lambda_{\alpha,n+1}\circ \Phi_{\alpha,n-2}(x)||=
            $$
            $$
            \SU_{x\in A}||\Phi_{\alpha,n-2}^{-1}\circ T(x)-\Phi_{\alpha,n-2}^{-1}\circ \Lambda_{\alpha,\sigma(n+1)}^{-1} \circ T \circ\Lambda_{\alpha,n+1}(x)||
            $$
            We will have that
            $$
            \SU_{x\in A}||\Phi_{\alpha,n-2}^{-1}\circ T(x)-\Phi_{\alpha,n-2}^{-1}\circ \Lambda_{\alpha,\sigma(n+1)}^{-1} \circ T \circ\Lambda_{\alpha,n+1}(x)||\leq
            $$
            $$
            \leq \SU_{x\in A}||\Phi_{\alpha,n-2}^{-1}\circ T(x)-\Phi_{\alpha,n-2}^{-1} \circ T \circ\Lambda_{\alpha,n+1}(x)||+\SU_{x\in A}||\Phi_{\alpha,n-2}^{-1} \circ T \circ\Lambda_{\alpha,n+1}(x)-\Phi_{\alpha,n-2}^{-1}\circ \Lambda_{\alpha,\sigma(n+1)}^{-1} \circ T \circ\Lambda_{\alpha,n+1}(x)||=
            $$
            $$
           \SU_{x\in A}||\Phi_{\alpha,n-2}^{-1}\circ T(x)-\Phi_{\alpha,n-2}^{-1} \circ T \circ\Lambda_{\alpha,n+1}(x)|| + \SU_{x\in A}||\Phi_{\alpha,n-2}^{-1}\circ \Lambda_{\alpha,\sigma(n+1)}^{-1}(x)-\Phi_{\alpha,n-2}^{-1}(x)||.
            $$
            Since $|\sigma(n+1)-(n+1)|\leq 2$ for all $n$, the first part of the result follows, the second part is similar.
        \end{proof}
        \begin{lemma}
            If $\D(A_{\alpha,n})\leq {\rm Min}\{\epsilon_{3,n},\epsilon_{1,n}\}$ for all $n$, then we have
            $$
            \SU_{x\in X}||(\Phi^{\sigma}_{\alpha,n})^{-1}\circ T \circ \Phi_{\alpha,n}(x)-(\Phi^{\sigma}_{\alpha,n+1})^{-1}\circ T \circ \Phi_{\alpha,n+1}(x)||\leq \frac{1}{2^{n-3}}
            $$
            and
            $$
             \SU_{x\in X}||(\Phi^{\tau}_{\alpha,n})^{-1}\circ T^{-1} \circ \Phi_{\alpha,n}(x)-(\Phi^{\tau}_{\alpha,n+1})^{-1}\circ T^{-1} \circ \Phi_{\alpha,n+1}(x)||\leq \frac{1}{2^{n-3}}.
            $$
        \end{lemma}
        \begin{proof}
            Note that we have defined $N_s$ and $\eta_s$ in the beginning of Theorem \ref{welldefinedness}. By condition (4), if $n-1\leq r\leq n+3$, $x\in X$ and $\Lambda_{\alpha,r}^{-1}(x)\neq x$, then we have that $x,\Lambda_{\alpha,r}^{-1}(x)\in N_{n-2} \cap A_{\alpha,r}$ and $||x-\Lambda_{\alpha,r}^{-1}(x)||\leq \eta_{n-2}$.

            If $T\circ \Lambda_{\alpha,n+1}(x)\neq T(x)$, then by the definition of $N_s$, we know that both $T\circ \Lambda_{\alpha.n+1}$ and $T(x)\in N_{n-2}$. Since $T$ is an isometry, we have
            \begin{equation} \label{Distance}
                ||T\circ \Lambda_{\alpha,n+1}(x)-T(x)||=||\Lambda_{\alpha,n+1}(x)-x||<\eta_{n-2}.
            \end{equation}
            By the choice of $\eta_s$ thus $\epsilon_{1,s}$ and $\epsilon_{3,s}$, for any $n-1\leq r\leq n+3$, we have
            $$
            \SU_{x\in X}||\Phi_{\alpha, n-2}^{-1}\circ \Lambda_{\alpha,r}^{-1}(x)-\Phi_{\alpha,n-2}^{-1}(x)||<\frac{1}{2^{n-2}}.
            $$
            Also, by (\ref{Distance}) and (iii), we have 
            $$
            \SU_{x\in X}||\Phi_{\alpha,n-2}^{-1}\circ T \circ \Lambda_{\alpha,n+1}(x)-\Phi^{-1}_{\alpha,n-2}\circ T(x)||< \frac{1}{2^{n-2}}.
            $$

            Thus, by Lemma \ref{controlnumber}, we have
            $$
            \SU_{x\in X}||(\Phi^{\sigma}_{\alpha,n})^{-1}\circ T \circ \Phi_{\alpha,n}(x)-(\Phi^{\sigma}_{\alpha,n+1})^{-1}\circ T \circ \Phi_{\alpha,n+1}(x)|| \leq \frac{1}{2^{n-3}}
            $$
            The proof of the second inequality is similar.
\end{proof}
\begin{lemma}
    $\{(\Phi^{\sigma}_{\alpha,n})^{-1}\circ T \circ \Phi_{\alpha,n}\}_n$ and $\{(\Phi^{\tau}_{\alpha,n})^{-1}\circ T^{-1} \circ \Phi_{\alpha,n}\}_n$ are uniformly continuous on $X$.
\end{lemma}
\begin{proof}
    By Lemma \ref{uniformcont}, we know that $(\Phi^{\sigma}_{\alpha,0})^{-1}\circ T \circ \Phi_{\alpha,0}$ is uniformly continuous.
    Assume inductively that $(\Phi^{\sigma}_{\alpha,n})^{-1}\circ T \circ \Phi_{\alpha,n}$ is uniformly continuous. Since functions $(\Phi^{\sigma}_{\alpha,n+1})^{-1}\circ T \circ \Phi_{\alpha,n+1}$ and $(\Phi^{\sigma}_{\alpha,n})^{-1}\circ T \circ \Phi_{\alpha,n}$ differ only on $[x_{n+1}-\epsilon_{1,n+1},x_{n+1}+\epsilon_{1,n+1}]\times \T=D$. It is enough to show that $(\Phi^{\sigma}_{\alpha,n+1})^{-1}\circ T \circ \Phi_{\alpha,n+1}$ is uniformly continuous on $D$. Note that $D$ is invariant under $\Lambda_{\alpha,m}$ for all $m\in \mathbb{N}$ and 
    $$
(\Phi^{\sigma}_{\alpha,n+1})^{-1}\circ T \circ \Phi_{\alpha,n+1} (D)=[x_{\sigma(n+1)}-\epsilon_{1,n+1},x_{\sigma(n+1)}+\epsilon_{1,n+1}]\times \T=D'.
    $$
    Note that $D'$ and $D$ are disjoint, also, $D'$ is invariant under $\Lambda_{\alpha,m}$ for all $m\in \mathbb{N}$. By condition (5), we know that  $\Phi_{\alpha,n}$ is uniformly continuous on $D$ and $(\Phi^{\sigma}_{\alpha,n})^{-1}$ is uniformly continuous on $B$. By Lemma \ref{uniformcont}, we know that $\Lambda_{\alpha.\sigma(n+1)}^{-1} \circ T \circ \Lambda_{\alpha, n+1}$ is uniformly continuous on $D$.

    Since
    $$
    (\Phi^{\sigma}_{\alpha,n+1})^{-1}\circ T \circ \Phi_{\alpha,n+1}=(\Phi^{\sigma}_{\alpha,n})^{-1}\circ \Lambda_{\alpha,\sigma(n+1)}^{-1}\circ T\circ \Lambda_{\alpha,n+1} \circ \Phi_{\alpha,n},
    $$
    we have that $(\Phi^{\sigma}_{\alpha,n+1})^{-1}\circ T \circ \Phi_{\alpha,n+1}$ is uniformly continuous as required.
\end{proof}
  \end{proof}
  First, by condition (4), we know that $\{\Phi^{\sigma}_{\alpha,n}\}_n$ and $\{\Phi^{\tau}_{\alpha,n}\}_n$ uniformly converge to $\Phi_{\alpha}$. By Theorem \ref{welldefinedness}, we know that   $\{(\Phi^{\sigma}_{\alpha,n})^{-1}\circ T \circ \Phi_{\alpha,n}\}_n$ and $\{(\Phi^{\tau}_{\alpha,n})^{-1}\circ T^{-1} \circ \Phi_{\alpha,n}\}_n$ are uniformly continuous on $X$. Denote the limit of  $\{(\Phi^{\sigma}_{\alpha,n})^{-1}\circ T \circ \Phi_{\alpha,n}\}_n$ and $\{(\Phi^{\tau}_{\alpha,n})^{-1}\circ T^{-1} \circ \Phi_{\alpha,n}\}_n$ by $\Gamma_{\alpha}$ and $\Delta_{\alpha}$ respectively. We have that $\Gamma_{\alpha}$ and $\Delta_{\alpha}$ are uniformly continuous on $X$. Thus, $S_{\alpha}$ is well-defined. Also, we have $\Phi_{\alpha}\circ \Gamma_{\alpha}=T\circ \Phi_{\alpha}$ and $\Phi_{\alpha}\circ \Delta_{\alpha}=T^{-1} \circ \Phi_{\alpha}$.  In a conclusion, $\Phi_{\alpha}$ is a factor map from $(\T^2,S_{\alpha})$ to $(\T^2,T)$.
  The last step of proving the well-definedness is to show the following:
  \begin{lemma}
      For any $z\in \T^2$, we have $\Phi_{\alpha}^{-1}(\{z\})$ is a singleton if $z\neq z_n$ for all $n\in \mathbb{N}$.
  \end{lemma}
  \begin{proof}
      Suppose $\Phi_{\alpha}(x)=\Phi_{\alpha}(y)$ for some $x\neq y$, we may assume that $||x-y||> \frac{1}{n}$ for some $n$.  Then
      $$
     0=||\Phi_{\alpha}(x)-\Phi_{\alpha}(y)||= {\rm lim}_m ||\Phi_{\alpha,m}(x)-\Phi_{\alpha,m}(y)|| \geq \prod_{r=n}^{\infty} (1-\frac{1}{a_r})||\Phi_{\alpha,n}(x)-\Phi_{\alpha,n+1}(y)||.
      $$
      Since $\prod_{r=n}^{\infty} (1-\frac{1}{a_r})$ is positive, we know that $\Phi_{\alpha,n}(x)=\Phi_{\alpha,n}(y)=z_i$ for some $i\leq n$. By condition (5), we know that $\Phi_{\alpha}(x)=\Phi_{\alpha}(y)=z_i$ for some $i\leq n$.
  \end{proof}

\section{Topological type of sequences} \label{tt section}
 Before we move to the Borel reduction part, we will define an equivalence relation which is related to topological conjugacy relation. It was initially defined in \cite[Definition 4.2]{BoLiminimal}.

 \begin{definition}
     Let $X$ be a compact metric space, we define an equivalence relation $E_{\rm tt}(X)$ on $X^\omega$:

     $$
     (x_n)E_{\rm tt}(y_n) \Leftrightarrow\,  (x_n), (y_n)\,\, \mbox{have the same convergent subsequences}.
     $$
\end{definition}
     
\begin{lemma} \label{tt type}
    Two minimal systems $(X,f)$ and $(Y,g)$ are conjugate if and only if there exists $x\in X$ and $y\in Y$ such that $(f^nx) E_{\rm tt} (g^ny)$. Also, if there is an isomorphism $h$ sending $x$ to $y$, then $(f^nx) E_{\rm tt} (g^nh(y))$.
\end{lemma}
\begin{proof}
   ($\Leftarrow$) On one hand, if two systems $(X,f)$ and $(Y,g)$ are conjugate by $h$, one can take $x$ and $h(x)$, then $(f^nx)$ and $g^n(h(x))$ have the same topological type. 
   
   ($\Rightarrow$) On the other hand, since $(X,f)$ is minimal, for any $a\in X$, there exists an increasing sequence of natural numbers $(n_k)$ such that $(f^{n_k}x)_k$ converges to $a$, we define $h(a)\in y$ to be the limit of $g^{n_k}(y)$. It is routine to check that $h$ is a well-defined conjugacy map.

\end{proof}
\section{Proof of the reduction from $c_0$ to $E_{\rm cm}$} \label{endgame}
   In this section, we analyze the conjugacy relation of $S_\alpha$. Also, in the rest of the paper, we will use $E_{\rm tt}$ for $E_{\rm tt}(\T^2)$. We start with some general analysis of conjugacy relation. 

   \begin{definition}
       Let $(X,f)$ be a system, two points $x,y\in X$ are called \textbf{non-proximal} if there exists a $\epsilon_0>0$ such that $\forall n\,\,d(f^nx,f^ny)>\epsilon_0$. A system is called \textbf{distal} if any two distinguished points are {non-proximal}.
   \end{definition}

   \begin{lemma}\label{factor map preserves distal pairs}
       Let $(X,f)$ and $(Y,g)$ be two dynamical systems and $\phi$ be a factor map from $(X,f)$ to $(Y,g)$, suppose $y$ and $y'$ are non-proximal in $(Y,g)$, then for any $x\in \phi^{-1}(\{y\}), x'\in \phi^{-1}(\{y'\})$, we have that $x$ and $x'$ are non-proximal in $(X,f)$.
   \end{lemma}
   \begin{proof}
    Suppose not, we can find a sequence of strictly increasing natural numbers $(n_k)$ such that ${\rm lim}_kd(f^{n_k}x,f^{n_k}x')=0$. By taking a subsequence, we may assume that both $(f^{n_k}x)_k$ and $(f^{n_k}x')_k$ are convergent to the same point $z$. But since $\phi$ is a factor map, this means
    $$
    \phi(z)=\mbox{lim}_k\phi(f^{n_k}x)=\mbox{lim}_k(g^{n_k}\phi(x))= \mbox{lim}_k(g^{n_k}y).
    $$
    similarly, we have
    $$
    \mbox{lim}_k(g^{n_k}y')=\mbox{lim}_k(g^{n_k}\phi(x'))=\mbox{lim}_k\phi(f^{n_k}x')=\phi(z).
    $$
    But since $y$ and $y'$ are non-proximal, ${\rm lim}_k (g^{n_k}y)$ can not be equal to ${\rm lim}_k (g^{n_k}y')$, a contradiction.
    
   \end{proof}

   \begin{definition}
       Let $(X,f)$ be a system, two points $x,y\in X$ are \textbf{asymptotic} if ${\rm lim}_n d(f^nx,f^ny)=0$.
   \end{definition}

       \begin{lemma} \label{distancewithsameimage}
           For any $\alpha\in [0,1]^{\omega}$. Suppose $\Phi_{\alpha}(x)=\Phi_{\alpha}(y)=z_k$, then we have
           $$
            ||x-y|| \leq \frac{1}{k}.
           $$
       \end{lemma}
       \begin{proof}
           By condition (2), we know that if $\Phi_{\alpha}(x)=\Phi_{\alpha}(y)=z_k$, then $||x-y||\leq \delta_{\alpha,k}$. And by the third item of condition (5), this implies that $||x-y||\leq \frac{1}{k}$.
       \end{proof}
        \begin{lemma}\label{asymptotic}
       For any $\alpha\in [0,1]^\omega$, in the minimal system $(\T^2,S_\alpha)$, two points $x\neq y$ are asymptotic if and only if $\Phi_\alpha(x)=\Phi_\alpha(y)=z_n$ for some $n\geq0$ otherwise they are non-proximal.
   \end{lemma}
   \begin{proof}
       \textbf{Case 1}. Suppose one of $\Phi_\alpha(x)$ and $\Phi_\alpha(y)$ is not equal to $z_n$ for all $n$. We may assume that $\Phi_{\alpha}(x)=z\neq z_n$ for all $n\in \mathbb{N}$. Then we know $\Phi_{\alpha}^{-1}(\{z\})=\{x\}$. Thus, we have $\Phi_{\alpha}(x)\neq \Phi_{\alpha}(y)$. Since $(\T^2,T)$ is a distal system, we know $\Phi_{\alpha}(x)$ and $\Phi_{\alpha}(y)$ are non-proximal. Then $x$ and $y$ are non-proximal by Lemma \ref{factor map preserves distal pairs}. 
       
       \textbf{Case 2}. For the other part, suppose $\Phi_\alpha(x)=\Phi_\alpha(y)=z_n$ for some $n$. Since $\Phi_{\alpha}$ will not influence the first coordinate, we assume that $x=(x_n,a)$ and $y=(x_n,b)$. By numerical condition (1), we know that actually,
       $$
       z_n=\Phi_{\alpha}(x_n,a)=\Phi_{\alpha,n}(x_n,a)=\Lambda_{\alpha,n}\circ\Phi_{\alpha,n-1}(x_n,a)=\Lambda_{\alpha,n}\circ\Phi_{\alpha,n-1}(x_n,b).
       $$
        By our construction, this means 
        $$
        \Phi_{\alpha,n-1}(x_n,a),\Phi_{\alpha,n-1}(x_n,b)\in \{x_n\}\times [y_n-\delta_{\alpha,n}/2,y_n+\delta_{\alpha,n}/2]
        $$
       Thus, 
       $$
        ||\Phi_{\alpha,n-1}(x_n,x)-\Phi_{\alpha,n-1}(x_n,y)||=|a-b|\leq \delta_n.
       $$
      by Lemma \ref{distancewithsameimage}, we have 
       $$
       |a-b|<1/n.
       $$
       Since $\Phi_{\alpha}S=T\Phi_{\alpha}$, we know $\Phi_{\alpha}(S_{\alpha}x)=\Phi_{\alpha}(S_{\alpha}y)=Tz_n=z_{\sigma(n)}$, similarly by Lemma \ref{distancewithsameimage}, this means
       $$
       ||S_\alpha x-S_{\alpha}y||<1/\sigma(n).
       $$
       Since $\Phi_{\alpha}$ is a factor map, we have $\Phi_{\alpha} (S^m_{\alpha}x)=\Phi_{\alpha}(S^m_{\alpha}y)=T^m\Phi_{\alpha}(x)=T^mz_{n}=z_{\sigma^m(n)}$, similarly by Lemma \ref{distancewithsameimage}, we have
       $$
       ||S^m_\alpha x-S^m_{\alpha}y||<1/\sigma^m(n).
       $$
       which converges to $0$ as $m$ goes to infinity.
   \end{proof}
   The following lemma will be used to distinguish two systems $(\T^2,S_{\alpha})$ and $(\T^2,S_{\beta})$ for $\alpha,\beta\in [0,1]^{\omega}$.
   
    \begin{lemma}\label{actually this is a pointed conjugacy}
        Let $\alpha,\beta\in 2^{\omega}$, take any $w_n,w_n'\in \T^2$ such that $\Phi_{\alpha}(w_n)=\Phi_{\beta}(w'_n)=T^nz_0=z_{2n}$, then two systems $(\T^2,S_{\alpha})$ and $(\T^2,S_{\beta})$ are topologically conjugate if and only if $(w_n)E_{\rm tt}(w_n')$.
    \end{lemma}
    \begin{proof}
        ($\Rightarrow$) Suppose $(\T^2,S_{\alpha})$ and $(\T^2,S_{\beta})$ are topologically conjugate by $f$. By Lemma \ref{asymptotic}, there exists $v_n\neq w_n$ (any other element in $\Phi_{\alpha}^{-1}(\{z_n\})$) such that $v_n$ and $w_n$ are asymptotic. Since conjugacy maps preserve asymptotic pairs, $f(v_n)$ and $f(w_n)$ are asymptotic, by Lemma \ref{asymptotic}, we know $\Phi_{\beta}f(w_n)=\Phi_{\beta}f(v_n)=z_k$ for some $k$. 
        We can choose $d\in \mathbb{Z}$ such that $T^dz_k=z_{2n}$. And since 
        $$
         \Phi_{\beta} \circ S_\beta=T \circ \Phi_{\beta},
        $$
        we have
        $$
        \Phi_{\beta}S_\beta^df(w_n)=T^d\Phi_{\beta}f(v_n)=T^dz_k=z_{2n}.
        $$

        In conclusion, we may assume that $\Phi_\beta f(w_n)=z_{2n}$ in addition to our assumption that $\Phi_{\alpha}(w_n)=z_{2n}$.
        
        Since $f$ is a conjugacy map from $(\T^2,S_{\alpha})$ to $(\T^2, S_{\beta})$ sending $w_n$ to $f(w_n)$, by Lemma \ref{tt type}, we have
        $$
        (S^m_{\alpha}w_n)_mE_{\rm tt}(S_{\beta}^mf(w_n))_m.
        $$
        Again, since $\Phi_{\alpha}$ is a factor map, for any $m\in \mathbb{Z}$, we have
        $$
       \Phi_{\alpha}(S_{\alpha}^mw_n)=T^m \Phi_{\alpha}(w_n)=T^mz_{2n}.
        $$
        By definition of $(w_n)$, we have $\Phi_{\alpha}(w_{n+m})=T^mz_{2n}$. By the above and Lemma \ref{asymptotic}, we know 
        $$
        ||S^m_\alpha w_{n}-w_{n+m}|| \rightarrow0
        $$
        as $m \rightarrow\infty$. Thus,
        $$
         (S^m_{\alpha}w_n)_mE_{\rm tt} (w_{n+m})_m
        $$
        and since we already assumed that $\Phi_\beta f(w_n)=z_{2n}$, by the same argument, we have
        $$
         (S^m_{\beta}f(w_n))_mE_{\rm tt} (w'_{n+m})_m.
        $$
        Thus, we have
        $$
        (w'_{n+m})_m E_{\rm tt} (w_{n+m})_m.
        $$
        But topological type does not depend on any finitely many terms in the beginning, this means 
        $$
         (w_n)_n E_{\rm tt} (w'_n)_n.
        $$
        ($\Leftarrow$) We can take $w_n$ to be $S_\alpha^n w_0$ for some $w_0\in \Phi^{-1}_{\alpha}(z_0)$ and $w_n'$ to be $S^n_\beta w_0'$ for some $w_0'\in \Phi^{-1}_{\beta}(z_0)$. Now we have
        $$
        (S^n_{\alpha}w_0)E_{\rm tt}(S^n_{\beta}w_0'),
        $$
        by Lemma \ref{tt type}, we are done.
    \end{proof}

    \begin{lemma}\label{norm close}
        For $\alpha,\beta\in [0,1]^{\omega}$, we have $\alpha (c_0)\beta$ implies that for any $\epsilon>0$ there exists $\delta>0$, such that for any $0<t<\delta$, we have $\SU_{x\in X}\{||(h^{\alpha}_t)^{-1}(x)-(h^{\beta}_t)^{-1}(x)||\}<\epsilon$.
    \end{lemma}
   \begin{proof}
     
     Fix $\alpha,\beta$ and $\epsilon>0$. Let $k\in \mathbb{N}$ be such that, for any $k'\geq k$, we have $|\alpha(k)-\alpha(k')|<4\epsilon$. Now let $\delta$ be $\epsilon_k$ and fix $0<t<\delta$. 
     
     If $t\in H^1_n$ for some $n\geq k$, note that we have $p_{\alpha}(t)=\frac{\alpha(n)}{4}$ and $p_{\beta}(t)=\frac{\beta(n)}{4}$. Then for any  $-1\leq s\leq1$,
     $$
      |(h^{\alpha}_t)^{-1}(s)-(h^{\beta}_t)^{-1}(s)|=|p_\alpha(t)-p_{\beta}(t)|=|\alpha(n)-\beta(n)|/4<4\epsilon/4=\epsilon.
     $$
     
     If $t\in H^3_n$, then we have $h^{\alpha}_t=h^{\beta}_t$, the inequality holds trivially. 
     
     If $t\in H^2_n$, then $p_{\alpha}(t)$ is between $\frac{\alpha(n)}{4}$ and $\frac{q_n}{4}$. Since $p_{\alpha}$ and $p_{\beta}$ are affine homeomorphisms, this means that there exists $\theta\in [0,1]$ such that 
     $$
      |p_{\alpha}(t)-p_{\beta}(t)|=\theta|\alpha(n)-\beta(n)|<4\epsilon/4=\epsilon.
     $$
     Similarly, when $t\in H^4_n$, the lemma holds.
\end{proof}

\begin{lemma} \label{label distance}
    Fix $\alpha\in [0,1]^{\omega}$, $0< t\leq 1$ and $x,y\in [-1,1]$. Then we have

    $$
     |(h_t^{\alpha})^{-1}(x)-(h^{\alpha}_t)^{-1}(y)|\leq \frac{(2+p_{\alpha}(t))}{t}|x-y|.
    $$
\end{lemma}
\begin{proof}
    Note that $(h_t^{\alpha})^{-1}$ is a polyline connecting the points
    $$
    (-1,-1),(-t,(-1-t)/2),(-t/2,-1/2),(t/4,-p_{\alpha}(t)/4),(t/2,1/2),(t,(1+t)/2), (1,1).
    $$

    \begin{center} 
    \begin{tikzpicture}
        \draw[thick] (-3,-2) -- (-1,-1) --(-0.5,-0.5)-- (0.2,-0.3) -- (0.5,0.5)--(1,1)--(3,2);
        
        \fill[black] (-3,-2) circle (2pt);
        \fill[black] (-1,-1) circle (2pt);
        \fill[black] (-0.5,-0.5) circle (2pt);
        \fill[black] (0.2,-0.3) circle (2pt);
        \fill[black] (0.5,0.5) circle (2pt);
        \fill[black]  (1,1) circle (2pt);
        \fill[black]  (3,2) circle (2pt);
        
           \node[left] at (-2,-2.5) {$(-1,-1)$};
           \node[left] at (-1,-1) {$(-t,\frac{-t-1}{2})$};
           \node[left] at (-0.3,-0.1) {$(t/4,-p_{\alpha}(t))$};
           \node[right] at (-0.2,-0.7)  {$(-\frac{t}{2},-\frac{1}{2})$};
           \node[right]  at (0.5,0.5){$(\frac{t}{2},\frac{1}{2})$};
           \node[right]  at (1,1) {$(t,\frac{1+t}{2})$};
           \node[right]  at (2,2.5)  {$(1,1)$};

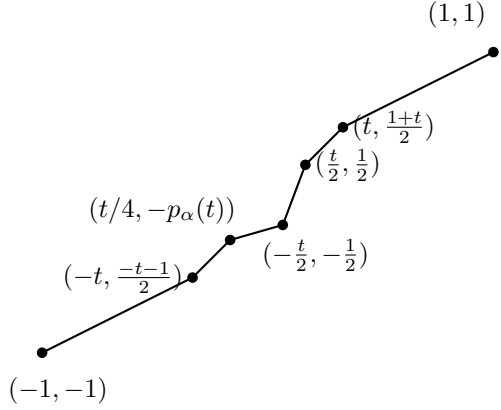
        \node[below] at (0,-3) {The function $(h^{\alpha}_t)^{-1}$};
    \end{tikzpicture}
\end{center}
    The slopes of those five lines are
    $$
    (2+p_{\alpha}(t))/3t,(2+p_{\alpha}(t))/t, (1-t)/2,1/2,1
    $$
    respectively. Clearly, we have
    $$
    (2+p_{\alpha}(t))/t>(2+p_{\alpha}(t))/3t
    $$
    and $(2+p_{\alpha}(t))/t>2>1$. Thus, the largest slope here is $(2+p_{\alpha}(t))/t$. Thus, we get the lemma.
\end{proof}
    
    \begin{theorem} \label{left direction of reduction}
        $\alpha (c_0)\beta$ implies $(\T^2,S_{\alpha})$ and $(\T^2,S_{\beta})$ are topologically conjugate.
    \end{theorem}
    \begin{proof}

        We pick points $w_n$ such that $\Phi_{\alpha}(w_n)=T^nz_0=z_{2n}$ and $w_n'$ such that $\Phi_{\beta}(w_n')=T^nz_0=z_{2n}$, by Theorem \ref{actually this is a pointed conjugacy}, we just have to show 
        $$
        \alpha (c_0) \beta \Rightarrow (w_n)E_{\rm tt} (w'_n).
        $$
        
        Now suppose $\alpha (c_0)\beta$. Then we will prove that $(w_n)E_{\rm tt} (w'_n)$.
        
        
        
        Suppose $(w_{r_k})$ is a convergent subsequence of $(w_n)$, we want to prove that $(w'_{r_k})$ is also convergent. We look at $\Phi_{\alpha}(w_{r_k})=\Phi_{\beta}(w_{r_k}')=T^{r_k}z_{0}=z_{2r_k}$, we know that $(z_{2r_k})$ is also convergent since $\Phi_{\alpha}$ is a continuous function defined on a compact metric space. 
        
        \textbf{Case 1.} Suppose $(z_{2r_k})$ converges to some $z$ such that $z\neq z_n$ for all $n$. In this case, we prove by contradiction, suppose that $(w'_{r_k})$ is not convergent. Then we can find two convergent subsequences of $(w'_{r_k})$ converging to $a$ and $b$ respectively with $a\neq b$. However, this means that $\Phi_{\beta}(a)=\Phi_{\beta}(b)=z$ which contradicts the fact that $\Phi_{\beta}^{-1}(\{z\})$ is a singleton (Condition (2)).  
        
        \textbf{Case 2.} Now suppose  $(z_{2r_k})$ converges to $z_m$ for some $m$.  
        By taking a small neighborhood $U$ of $z_m$, we may assume that all $\Lambda_{\alpha,n}$ for all $n<m$ and $\Lambda_{\beta,n}$ for all $n<m$ are homeomorphisms when restricted to $U$. 
        
        Now let $N_n$ is the least natural number greater than $m$ such that $z_{2{r_n}}\in A_{\alpha,N_n}$. We have that $N_n$ goes to infinity as $n$ goes to infinity. Indeed, by the first bullet of condition (5), for any $N\in \mathbb{N}$, we can find a very small neighborhood $V$ of $z_m$ such that for all $n$ with $n<N$ and $n\neq m$, $z_n\not \in V$.   Similarly, let $M_n$ be the least natural number greater than $m$ such that $z_{2r_n}\in A_{\beta,M_n}$. We also have that $M_n$ goes to infinity as $n$ goes to infinity. Now we take $f$ and $g$ to be the restriction of $\Phi_{\alpha,m-1}$ and $\Phi_{\beta,m-1}$ on $U$ respectively. 
        
        By the definition of $\Phi_{\alpha}$, we know
        $$
        \Phi_{\alpha}(w_{r_n}) =\Lambda_{\alpha,2r_n}\circ\cdots\Lambda_{\alpha,N_n}\circ\Lambda_{\alpha,m}(f(w_{r_n}))=z_{2r_n}=(x_{2r_n},y_{2r_n})
        $$
        and similarly,
        $$
        \Phi_{\beta}(w_{r_n}')=\Lambda_{\beta,2r_n}\circ\cdots\Lambda_{\beta,M_n}\circ\Lambda_{\beta,m}(g(w'_{r_n}))=z_{2r_n}=(x_{2r_n},y_{2r_n}).
        $$
        Note that this implies 
        $$
        \Lambda_{\alpha,2r_n-1} \circ\cdots\circ\Lambda_{\alpha,m}(f(w_{r_n}))\in [z_{r_n}-\delta_{\alpha,r_n}/2,z_{r_n}+\delta_{\alpha,r_n}/2].
        $$
        and
        $$
        \Lambda_{\beta,2r_n-1} \circ\cdots\circ\Lambda_{\beta,m}(g(w'_{r_n}))\in [z_{r_n}-\delta_{\beta,r_n}/2,z_{r_n}+\delta_{\beta,r_n}/2].
        $$
        Let $f(w_{r_n})=(x_{2r_n},a_n)$ and $g(w'_{r_n})=(x_{2r_n},b_n)$.
        But since the choice of $w_{r_n}\in \Phi_{\alpha}^{-1}(z_{2r_n})$ and $w'_{r_n}\in \Phi_{\beta}^{-1}(z_{2r_n})$ can be arbitrary,  we may assume that
        $$
        (x_{2r_n},a_n)=f(w_{r_n})= \Lambda_{\alpha,m}^{-1} \circ \Lambda_{\alpha,N_n}^{-1} \circ \cdots \Lambda^{-1}_{\alpha,2r_n-1}((x_{2r_n},y_{2r_n}))=\Lambda_{\alpha,m}^{-1} \circ (x_{2r_n},F_{\alpha,n}(y_{2r_n}))
        $$
        and
        $$
       (x_{2r_n},b_n)=g(w'_{r_n})= \Lambda_{\beta,m}^{-1} \circ \Lambda_{\beta,M_n}^{-1} \circ \cdots \Lambda^{-1}_{\beta,2r_n-1}(x_{2r_n},y_{2r_n})=\Lambda_{\beta,m}^{-1} \circ(x_{2r_n},G_{\beta,n}(y_{2r_n}))
        $$
       
        Now we let $j_n=|x_m-x_{2r_n}|$. Also, let $t_n=|G_{\beta,n}(y_{2r_n})-y_m|$ and $t'_n=|F_{\alpha,n}(y_{2r_n})-y_m|$.
        
        By the definition of $\Lambda_{\alpha,m}$ can simply write
        $$
         (x_{2r_n},a_n)=\Lambda_{\alpha,m}^{-1} \circ (x_{2r_n}, y_m+t_n)=(x_{2r_n},y_m+h^\alpha_{\delta_{\alpha,m}}(h^\alpha_{j_n})^{-1}t_n)
        $$
        and 
        $$
         (x_{2r_n},b_n)=\Lambda_{\beta,m}^{-1} \circ (x_{2r_n},y_m+t_n')=(x_{2r_n}, y_m+h^\beta_{\delta^\beta_m}(h^\beta_{j_n})^{-1}t_n').
        $$
        
        Now we prove $((h^\beta_{j})^{-1}(t_n'))_n$ convergent.
         By our assumption, we know
         \begin{equation} \label{D1}
             |t_n-t'_n|<4\delta_{L_n}
         \end{equation}

         where $\delta_{L_n}={\rm max}\{\delta_{\beta,M_n},\delta_{\alpha,N_n}\}$. Since $z_{2r_n}\in A_{\alpha,N_n}$ and $z_{2r_n}\in A_{\beta,M_n}$, we know 
         $$
         |x_{2r_n}-x_{N_n}|<\delta_{\alpha, N_n}
         $$
         and
         $$
         |x_{2r_n}-x_{M_n}|<\delta_{\beta,M_n}
         $$
         thus by condition (5), we have 
         $$
          |x_{N_n}-x_m|>N_n\delta_{\alpha,N_n}.
         $$
         and
         $$
           |x_{M_n}-x_m|>M_n\delta_{\beta,M_n}.
         $$
         Thus, we have 
         $$j_n>(N_n-1)\delta_{\alpha,N_n}
         $$ 
         and 
         $$
         j_n>(M_n-1)\delta_{\beta,M_n}
         $$
         By Lemma \ref{label distance} and (\ref{D1}), we know
         
         $$
         |(h^\alpha_{j_n})^{-1}t_n-(h^\alpha_{j_n})^{-1}t'_n|< \frac{(2+p_{\alpha}(t))4\delta_{L_n}}{j_n}<\frac{4(2+p_{\alpha}(t))}{L_n-1}
         $$
         which converges to $0$ as $n$ goes to infinity.
         Since $(w_{r_n})$ converges, we know $(f(w_{r_n}))_n$ converges and thus $(a_n)_n$ converges. This means $h^{\alpha}_{\delta_{\alpha,m}}(h^{\alpha}_{j_n})^{-1}(t_n)$ converges. Since $(h^{\alpha}_{\delta_{\alpha,m}})$ is a homeomorphism, this implies that $((h^\alpha_{j_n})^{-1}(t_n))_n$ converges and by the above formula, $((h^\alpha_{j_n})^{-1}(t'_n))_n$ converges.
         Now, by Lemma \ref{norm close}, for any $\epsilon>0$, we can find $\delta>0$ for any $0<t<\delta$, we have $||(h^{\alpha}_t)^{-1}-(h^{\beta}_t)^{-1}||_{\infty}<\epsilon$. Take $N$ large enough such that $|x_{2r_n}-x_m|<\delta$ for all $n>N$. Also, since $((h^\alpha_{j_n})^{-1}(t_n))_n$ converges we can assume that for any $n,n'>N$, we have
         $$
         |(h^\alpha_{j_n})^{-1}(t_n)-(h^\alpha_{j_{n'}})^{-1}(t_{n'})|<\epsilon.
         $$
         Now, for all $n,n'>N$, we have
         $$
         |(h^\beta_{j_n})^{-1}(t'_n)-(h^\beta_{j_{n'}})^{-1}(t'_{n'})|\leq 
         $$
         $$
         |(h^\beta_{j_n})^{-1}(t'_n)-(h^\alpha_{j_n})^{-1}(t'_{n})|+|(h^\alpha_{j_n})^{-1}(t'_n)-(h^\alpha_{j_n'})^{-1}(t'_{n'})|+|(h^\alpha_{j_{n'}})^{-1}(t'_{n'}),(h^\beta_{j_{n'}})^{-1}(t'_{n'})|\leq
         $$
         $$
         ||(h^{\alpha}_{j_n})^{-1}-(h^{\beta}_{j_n})^{-1}||_{\infty}+||(h^{\alpha}_{j_{n'}})^{-1}-(h^{\beta}_{j_{n'}})^{-1}||_{\infty}+|(h^\alpha_{j_n})^{-1}(q_n)-(h^\alpha_{j_{n'}})^{-1}(q_{n'})|<3\epsilon
         $$
         when $n$ and $n'$ are large enough. Thus, $((h^\beta_{j_n})^{-1}(t'_n))_n$ converges, this means the sequence
         $$
         (h^\beta_{\delta_{\beta,m}}(h^\beta_{j})^{-1}(t'_n))_n
         $$ 
         converges, since $h^\beta_{\delta_{\beta,m}}$ is a homeomorphism. Thus, we know $b_n$ converges and this implies $g(w_{r_n}')$ thus $(w_{r_n}')$ converges.

         The other direction is symmetric.
         \end{proof}
        \begin{theorem} \label{Right direction of the reduction}
            $\alpha$ is not $c_0$ equivalent with $\beta$ implies $(\T^2,S_{\alpha})$ and $(\T^2,S_{\beta})$ are not topologically conjugate.
        \end{theorem}
     \begin{proof}
         We take a sequence $(w_{n})$ such that $\Phi_{\alpha}(w_n)=T^nz_0=z_{2n}$ and a sequence $(w'_n)$ such that $\Phi_{\beta}(w'_n)=T^nz_0=z_{2n}$. Suppose that $\alpha$ is not $c_0$-equivalent with $\beta$.  By Lemma \ref{actually this is a pointed conjugacy}, we will show that $(w_n)$ is not $E_{\rm tt}$ equivalent with $(w_n')$. Recall that $(q_n)$ is the fixed dense subset of $[0,1]$. We may assume that there exists increasing sequence of natural numbers $(n_k)$ and $(m_k)$ such that $\alpha(n_k)$ and $(q_{m_k})$ converges to $a\in [0,1]$ but $\beta(n_k)$ converges to $b\neq a$ as $k$ goes to infinity. 

         Since the rotation orbit of $z_0$ is dense, we can pick a sequence of $(z_{2r_k})=(x_{2r_k},y_{2r_k})$ with $e^3_{m_k}<\delta_{\alpha,0}$ such that 
         
         \begin{itemize}
             \item [(P1)] $d^3_{m_k}<|x_0-x_{2r_k}|<e^3_{m_k}$. 
             \item [(P2)]$||y_0-y_{2r_k}|-\frac{ |x_0-x_{2r_k}|}{4}|<\frac{|x_0-x_{2r_k}|}{2^k}$.
         \end{itemize}
         
          Now we pick another sequence $(z_{2s_k})=(x_{2s_k},y_{2s_k})$ with  $e^1_{n_k}<\delta_{\alpha,0}$ such that 
          \begin{itemize}
              
              \item [(Q1)]$d^1_{n_k}<|x_{2s_k}-x_0|<e^1_{n_k}$. 
              \item [(Q2)]$||y_{2s_k}-y_{0}|-\frac{ |x_0-x_{2s_k}|}{4}|< \frac{|x_0-x_{2s_k}|}{2^k}$.
          \end{itemize}
            Conditions ${\rm (P1)}$ and ${\rm (Q1)}$ imply that $(x_{2s_k})$ and $(x_{2r_k})$ converges to $x_0$ as $k$ goes to infinity. 

            We can assume that $N_k$ is the least natural number greater than $0$ such that $w_{r_k}\in A_{\alpha,N_k}$. Note that $N_k$ goes to infinity by the first bullet of condition (5).

            Since for all $l>2r_k$, by the first bullet of condition (5) , $\Lambda_{\alpha,l}$ is identity on $z_{2r_k}$, hence by the choice of $N_k$, we have 
            $$
            \Phi_{\alpha}(w_{r_k})=\Lambda_{\alpha,2r_k}\circ \cdots \Lambda_{\alpha,N_k} \circ \Lambda_{\alpha,0}(w_{r_k}).
            $$
             Pick $a_k$ such that

            $$
             w_{r_k}=(\Lambda_{\alpha,0})^{-1}\circ (\Lambda_{\alpha,N_k})^{-1}\circ \cdots \circ (\Lambda_{\alpha,2r_k-1})^{-1}(a_k).
            $$

            Let $$
            (\Lambda_{\alpha,N_k})^{-1}\circ \cdots \circ (\Lambda_{\alpha,2r_k-1})^{-1}(a_k)= w_{r_k}= (x_{2r_k},b_k).
            $$  
            And let $i_k=|x_{2r_k}-x_0|$ and $j_k=|b_k-y_0|$, by conditions $({\rm P2})$, $({\rm Q2})$ and the fact that $\Lambda_{\alpha,N_k}$ is identity except on $A_{\alpha,N_k}$, we have
            \begin{equation} \label{XYRELA}
                |j_k-i_k/4|<|j_k-|y_0-y_{2r_k}||+||y_0-y_{2r_k}|-i_k/4|<2\delta_{\alpha,N_k}+i_k/2^k.
            \end{equation}
            
            By the definition of $\Lambda_{\alpha,0}$, we know
            $$
            (\Lambda_{\alpha,0})^{-1}(x_{2r_k},b_k)=(x_{2r_k},h^\alpha_{\delta_{\alpha,0}}(h^{\alpha}_{i_k})^{-1}j_k)
            $$
            and
            $$
             (\Lambda_{\alpha,0})^{-1}(x_{2r_k},i_k/4)=(x_{2r_k},h^\alpha_{\delta_{\alpha,0}}(h^{\alpha}_{i_k})^{-1}(i_k/4)).
            $$
            By Lemma \ref{label distance} and the fact that $d^3_{m_k}<i_k<e^3_{m_k}$, we have
            $$
            |(h^{\alpha}_{i_k})^{-1}(i_k/4)-(h^{\alpha}_{i_k})^{-1}j_k| \leq 
            \frac{|j_k-i_k/4|(2+q_{m_k})}{i_k}
            $$
            By (\ref{XYRELA}), we have that 
            $$
            \frac{|j_k-i_k/4|(2+q_{m_k})}{i_k}\leq \frac{[i_k/2^k+2(\delta_{N_k})](2+q_{m_k})}{i_k}.
            $$
            By the last bullet of condition (5), we know that $\delta_{\alpha,N_k}<\frac{1}{N_k}|x_{2N_k}-x_0|$. Since $x_{2r_k}\in A_{\alpha,N_k}$, we have that $|x_{2r_k}-x_{2N_k}|<\delta_{\alpha,N_k}$. Thus, we have 
            \begin{equation}
                \frac{\delta_{\alpha,N_k}}{i_k}\leq \frac{|x_{2N_k}-x_0|}{N_ki_k}\leq \frac{i_k+\delta_{\alpha,N_k}}{N_ki_k}.
            \end{equation}
            which implies that
            \begin{equation} \label{XYcontrol(5)}
                 \frac{\delta_{\alpha,N_k}}{i_k} \leq  \frac{1}{N_k-1}.
            \end{equation}
           By (\ref{XYRELA}) and (\ref{XYcontrol(5)}), we have
             $$
           \frac{|i_k-j_k/4|(2+q_{m_k})}{i_k}\leq \frac{[i_k/2^k+2(\delta_{N_k})](2+q_{m_k})}{i_k} <(
            1/2^k+\frac{2}{N_1-1})(2+a(n_k)),
            $$
            which converges to $0$. This means $w_{r_k}=(x_{2r_k},b_k)$ converges to the same limit as $(x_{2r_k}, h^{\alpha}_{\delta_{\alpha,0}}(h^\alpha_{i_k})^{-1}{(i_k}/4))$ does. 
            But since $d^3_{m_k}<i_k<e^3_{m_k}$, we have $(h^{\alpha}_{i_k})^{-1}(i_k/4)=q_{m_k}/4$. Thus,
            $$
            (x_{2r_k},h^\alpha_{\delta_0}(h^{\alpha}_{i_k})^{-1}(i_k/4))=(x_{2r_k},h^{\alpha}_{\delta_0}(q_{m_k}))
            $$
            which converges to $(x_0,h^{\alpha}_{\delta_{\alpha,0}}(a/4))$.  This means $w_{r_k}$ converges to $(x_0,h^{\alpha}_{\delta_{\alpha,0}}(a/4))$. 
            Similarly, one can prove that $w_{s_k}$ also converges to $(x_0,h^{\alpha}_{\delta_{\alpha,0}}(a/4))$.  
            Now, for $(\T^2,S_{\beta})$, we can similarly pick up two sequences of  points $(w'_{r_k})$ and $(w'_{s_k})$ such that $\Phi_{\beta}(w'_{r_k})=z_{2r_k}$ and $\Phi_{\beta}(w'_{s_k})=z_{2s_k}$ satisfying conditions ${\rm (P1)}$, ${\rm (P2)}$ and ${\rm (Q1)}$,${\rm (Q2)}$, respectively. By the same argument, we can prove that
            $(w'_{r_k})$ converges to $(x_0,h^{\beta}_{\delta_{\beta,0}}(a/4))$ and $w'_{s_k}$ converges to $(x_0,h^{\beta}_{\delta_{\beta,0}}(b/4))$. This means the sequence obtained by alternating $(w_{r_k})$ and $(w_{s_k})$ is convergent. However, the sequence obtained by alternating $(w'_{r_k})$ and $(w'_{s_k})$ is divergent since they converge to different limits. By Lemma \ref{tt type}, we  know $(\T,S_{\alpha})$ and $(\T^2,S_{\beta})$ are not conjugate.
            
     \end{proof}
   
\bibliography{bibliography}     
\end{document}